\documentclass{amsart}
\usepackage{amscd,amssymb,amsopn,amsmath,amsthm,graphics,amsfonts,accents,enumerate,verbatim,calc}
\usepackage[dvips]{graphicx}
\usepackage[colorlinks=true,linkcolor=red,citecolor=blue]{hyperref}
\usepackage[all]{xy}

\usepackage{tikz}

\calclayout

\newcommand{\rt}{\rightarrow}
\newcommand{\lrt}{\longrightarrow}

\newcommand{\st}{\stackrel}

\newcommand{\CE}{\mathcal{E}}

\newcommand{\CI}{\mathcal{I} }

\newcommand{\CM}{\mathcal{M} }

\newcommand{\CQ}{\mathcal{Q} }

\newcommand{\CY}{\mathcal{Y} }
\newcommand{\CV}{\mathcal{V}}
\newcommand{\BE}{\mathbf{E}}

\newcommand{\BP}{\mathbf{P}}

\newcommand{\Bs}{\mathbf{s}}
\newcommand{\Bt}{\mathbf{t}}

\newcommand{\N}{\mathbb{N} }

\newcommand{\Proj}{{\rm Prj} }

\newcommand{\Inj}{{\rm Inj}}

\newcommand{\RMod}{R \text{-} {\rm{Mod}}}

\newcommand{\RQIMod}{R{\CQ / \langle I \rangle} \text{-} {\rm{Mod}}}

\newcommand{\Rep}{{\rm Rep}}

\newcommand{\fd}{{\rm{fd}}}

\newcommand{\Coker}{{\rm{Coker}}}
\newcommand{\Ker}{{\rm{Ker}}}

\newtheorem{theorem}{Theorem}[section]
\newtheorem{corollary}[theorem]{Corollary}
\newtheorem{lemma}[theorem]{Lemma}
\newtheorem{proposition}[theorem]{Proposition}

\theoremstyle{definition}

\newtheorem{example}[theorem]{Example}

\newtheorem{remark}[theorem]{Remark}
\newtheorem{s}[theorem]{}

\theoremstyle{plain}
\newtheorem{stheorem}{Theorem}[subsection]

\newtheorem{scorollary}[stheorem]{Corollary}

\theoremstyle{definition}

\numberwithin{equation}{section}

\begin{document}

\title[Minimal Injective Resolutions and Auslander-Gorenstein Property]{Minimal Injective Resolutions and Auslander-Gorenstein Property for Path Algebras}

\author[Asadollahi, Hafezi and Keshavarz]{J. Asadollahi, R. Hafezi and M. H. Keshavarz}

\address{Department of Mathematics, University of Isfahan, P.O.Box: 81746-73441, Isfahan, Iran and School of Mathematics, Institute for Research in Fundamental Science (IPM), P.O.Box: 19395-5746, Tehran, Iran}
\email{asadollahi@ipm.ir \ \ and  \ \ asadollahi@sci.ui.ac.ir}

\address{Department of Mathematics, University of Isfahan, P.O.Box: 81746-73441, Isfahan, Iran}
\email{keshavarz@sci.ui.ac.ir}

\address{School of Mathematics, Institute for Research in Fundamental Sciences (IPM), P.O.Box: 19395-5746, Tehran, Iran }
\email{hafezi@ipm.ir}



\thanks{This research was in part supported by a grant from IPM (No. 93130216)}

\begin{abstract}
Let $R$ be a ring and $\CQ$ be a finite and acyclic quiver. We present an explicit formula for the injective envelopes and projective precovers in the category $\Rep (\CQ ,R)$ of representations of $\CQ$ by left $R$-modules. We also extend our formula to all terms of the minimal injective resolution of $R\CQ$. Using such descriptions, we study the Auslander-Gorenstein property of path algebras. In particular, we prove that the path algebra $R\CQ$ is $k$-Gorenstein if and only if $\CQ=\overrightarrow{A_{n}}$ and $R$ is a $k$-Gorenstein ring, where $n$ is the number of vertices of $\CQ$.
\end{abstract}

\maketitle

\tableofcontents

\section{Introduction}\label{1}
The theory of representations of quivers was initiated with the purpose of classifying finite dimensional algebras of finite representation type. Gabriel in \cite{G1} and \cite{G2} gave an explicit construction of indecomposable modules over a finite dimensional algebra, and in his work, he found the connection between the Dynkin diagrams of semisimple Lie algebras and the representation theory of algebras. After this connection was found, many authors have started the study of this theory of representation of quivers.

The classical representation theory of quivers considers finite quivers and assume that the base ring is algebraically closed field and that all vector spaces involve are finite dimensional (cf. \cite{Le}). In the recent years there has been a growing interest in the study of representations of quivers over general rings (not just fields) and this paper contains some results
in this direction. In fact, it should be considered as a continuation of the projects initiated in \cite{EH} and continued in \cite{EOT, EE1, EEG, AEHS}, to develop new techniques to study these more general representations.

\vspace{2mm}
In the first part of this paper, we recall the left and right adjoints of evaluation functors and their descriptions. We then apply them to describe injective envelopes and projective precovers of representations of finite acyclic quivers. Moreover, we describe explicitly, the terms of the minimal injective resolution of path algebra $R\CQ$, whenever $\CQ$ is a finite and acyclic quiver. Based on these descriptions, among other results, we study Auslander-Gorenstein property of path algebras. To see some results on Auslander-Gorenstein property of algebras see e.g. \cite{FGR, AR1, AR2}.

Let us be more precise. Let $M$ be a representation of $ \CQ $ by left $R$-modules. In Theorem \ref{Main I}, we provide an explicit formula for $\BE(M)$, the injective envelope of $M$ and $\BP(M)$, a projective precover of $M$.  In Theorem \ref{Main II} we extend our formula to all terms of the minimal injective resolution of $R\CQ$.

Generalized Nakayama Conjecture, $\mathbf{GNC}$ for short, says that each indecomposable injective module is a summand of some terms in a minimal injective resolution of an artin algebra.
As a corollary, we show that if $\mathbf{GNC}$ is true for $R$, then $\mathbf{GNC}$ is true for $R\CQ$, where $\CQ$ is a finite and acyclic quiver, see Corollary \ref{GNC} below.

Auslander introduced the notion of $k$-Gorenstein algebras, see e.g. \cite{FGR}. Let $R$ be a two-sided Noetherian ring. $R$ is called $k$-Gorenstein if $\fd (I^i)\leq i$ for every $0 \leq i \leq k-1$, where $I^i$ denotes the $i$th term of the minimal injective resolution of $R$ considered as a left $R$-module. Note that the notion of a $k$-Gorenstein ring is left-right symmetric. In Section \ref{5}, we apply our formulas to show that if  $\CQ$ is a finite, connected and acyclic quiver with $n$ vertices, then the path algebra $R\CQ$ is a $k$-Gorenstein ring if and only if $\CQ=\overrightarrow{A_{n}}$ and $R$ is a $k$-Gorenstein ring, see Theorem \ref{Main III}.
Also as an interesting example, show that when we consider the category of representations of a quiver over a ring, not necessarily a field, we in fact are studying the category of representations of a quiver with relations over a field, see Example \ref{5.6}.

We also turn our attention to the tensor product of path algebras and show that if $ A=K\CQ $ and $ B= K\CQ'$ are path algebras with respect to finite, connected and acyclic quivers $\CQ$ and $ \CQ'$, where $ K $ is a field, then $A\otimes_{K} B$ is $k$-Gorenstein if and only if $\CQ$ and $\CQ'$ are linear quivers.

All rings considered in this paper are associative with identity. The letter $R$ will usually denote such ring. All modules are left unitary $R$-modules. $\RMod$ will denote the category of left $R$-modules.

\section{Preliminaries}\label{2}
In this section, for the convenience of the reader, we present definitions and results that will be used throughout the paper.

\s{\sc Quivers and their representations.}\label{2.1}
A quiver $\mathcal{Q}$ is a directed graph. It will be denoted by a quadruple $\mathcal{Q} = (V, E, s, t)$, where $V$ and $E$ are respectively the sets of vertices and arrows of $\mathcal{Q}$ and $ s, t : E \rightarrow V$ are two maps which associate to every arrow $ \alpha \in E $ its source $ s(\alpha) $ and its target $ t(\alpha) $, respectively.
We usually denote the quiver $\mathcal{Q} = (V, E, s, t)$ briefly by $\mathcal{Q} = (V, E)$ or even simply by $\mathcal{Q}$. A vertex $v \in V$ is called a sink if there is no arrow $ \alpha $ with $ s(\alpha)=v $. $v$ is called a source if there is no arrow $ \alpha $ with $ t(\alpha)=v $.
A quiver $\mathcal{Q}$ is said to be finite if both $V$ and $E$ are finite sets. Throughout the paper we assume that $\mathcal{Q}$ is a finite quiver.

A path of length $ l\geqslant 1 $ with source $a$ and target $b$ (from $a$ to $b$) is a sequence of arrows $ \alpha_{l} \cdots \alpha_{2}\alpha_{1}$, where $\alpha_{i}\in E$, for all $ 1\leq i \leq l $, and we have $ s(\alpha_{1}) = a, s(\alpha_{i})=t(\alpha_{i-1})$ and $ t(\alpha_{l})=b $. If $ p=\alpha_{l} \cdots \alpha_{2}\alpha_{1} $ is a path of $ \CQ $ we extend the notation and let $ \Bs(p)=s(\alpha_1) $ and $ \Bt(p)=t(\alpha_l) $.
A path of length $ l\geqslant 1 $ is called a cycle if its source and target coincide. $\mathcal{Q}$ is called acyclic if it contains no cycles.

As Enochs et al. in \cite{EOT}, we can exploit induction to build a partition for $ V, $ the set of vertices of acyclic quiver $ \CQ$. Put
$ V_0 = \{ v \in V : \nexists \ \alpha \in E {\rm \ such \ that} \ s(\alpha)=v \}. $
Suppose $ n\geq 0 $ and we have defined $ V_i $ for all $ i \leq n $. Let
$ V_{n+1}=  \{ v \in V\setminus \bigcup_{i=0}^{n}V_i : \nexists \  \alpha \in \CE_n {\rm \ such \ that} \ s(\alpha)=v \}, $
where $ \CE_n =E\setminus \{ \alpha : t(\alpha) \in \bigcup_{i=0}^{n}V_i \} $. Dually one can define
$ V'_0 = \{ v \in V : \nexists \ \alpha \in E {\rm \ such \ that} \ t(\alpha)=v \}. $
And, if  $ V'_i $ is defined for every $ i \leq n $, then put
$ V'_{n+1}=  \{ v \in V\setminus \bigcup_{i=0}^{n}V_i : \nexists \ \alpha \in \CE'_n {\rm \ such \ that} \ t(\alpha)=v \}, $
where $ \CE'_n =E \setminus \{ \alpha : s(\alpha) \in \bigcup_{i=0}^{n}V'_i \} $.

For a fixed vertex $ v \in V $, the set of all $ w \in V $ with an arrow $ v \lrt w $ will be denoted by $V_{ \Bs(v)} $. Also, the set of all $ w \in V $ with an arrow $ w  \lrt v  $ will be denoted by $V_{ \Bt(v)} $. Similarly, $ E_{\Bs(v)} $, resp. $ E_{\Bt(v)} $, denotes the set of all arrows with initial, resp. terminal, vertex $v$.

A quiver $\mathcal{Q}$ can be considered as a category whose objects are the vertices of $\mathcal{Q}$ and morphisms are all paths in $\mathcal{Q}$. Assume that $R$ is a ring.
A representation $X$ of  $\CQ$ by $R$-modules  is a covariant functor $X : \CQ \lrt \RMod$.
Such a representation is determined by giving a module $X_{v}$ to
each vertex $v$ of $\CQ$ and a homomorphism $X_{\alpha} : X_{v} \lrt X_{w}$ to each
arrow $\alpha : v \lrt w$ of $\CQ$. And so, if $ p=\alpha_{l} \cdots \alpha_{2}\alpha_{1} $ is a path of $ \CQ $, then $ X_{p}=X_{{\alpha}_{l}} \cdots X_{{\alpha}_{2}}X_{{\alpha}_{1}} $.
A morphism between two representations $X$
and $Y$ is a natural transformation. Thus the representations of a quiver $ \CQ$ by modules over a ring $R$ form
a category, denoted by $\Rep (\CQ,R)$ or $ (\CQ, \RMod)$. This is a Grothendieck category with enough projectives and injectives. It is known that the category $(\CQ,\RMod)$ is equivalent to the category of modules over
the path ring $R\CQ.$

\s{\sc Left and right path spaces.}\label{2.3}
By the (left) path space of $\CQ,$ we mean the quiver $P(\CQ)$ whose vertices are the paths $p$ of $\CQ$ and whose arrows are the pairs $(p, \alpha p) : p \lrt \alpha p,$ where $p$ is a path of $\CQ$ and $\alpha$ is an arrow of $\CQ$ such that $\alpha p$ is defined.
It is clear then that $P(\CQ)$ is a forest. If $v$ is a vertex of $\CQ$ we let $P(\CQ)_v$ denote the
subtree of $P(\CQ)$ containing all paths of $\CQ$ with initial vertex $v$. If $p$ and $q$ are paths of $\CQ$
such that $qp$ is defined, we extend the notation and let $(p, qp): p \lrt qp$ denote a path of $P(\CQ)$. If $v$ is a vertex of $\CQ$ we let $P(\CQ)_v$ denote the subquiver of $P(\CQ)$ containing all paths of $ R\CQ$ with initial vertex $v$. For simplicity, the set of vertices of $P(\CQ)_v$
will be denoted by $ \CY_{v} $. Note that there is also an obvious definition of a right path space of $\CQ$.

\s{\sc Evaluation functor and its adjoints.}\label{2.4}
Associated to $v$, there exists a functor $e^{v} : \Rep (\CQ, R) \rightarrow \RMod$, called the evaluation functor, which assigns to every representation $X$ of $ \CQ $ its module at vertex $v$, denoted $X_{v}$.
It is proved in \cite{EH, EGOP} that
$ e^{v} $ possesses a right and also a left adjoint $ e^{v}_{\rho} $
and $ e^{v}_{\lambda} $, respectively.
In fact, $ e^{v}_{\rho}: \RMod \rt \Rep (\CQ, R) $ is defined as follows:
let $M$ be an arbitrary module in $\RMod$.
Then $ e^{v}_{\rho}(M)_{w}=\bigoplus_{\CQ(w,v)} M $, where $\CQ(w, v)$ denotes the set of
paths starting in $w$ and terminating in $v$. The maps are natural projections.
The left adjoint of $ e^{v} $ is defined similarly:
for any $R$-module $M$, one defines $ {e^{v}_{\lambda}(M)}_{w}=\bigoplus_{\CQ(v, w)}M $. The maps are natural injections.

\begin{remark}
Based on the properties of these adjoints, one may deduce that for any projective module $ P \in\Proj (R)$, the representation $ e^{v}_{\lambda}(P)$ is a projective representation
of $\CQ$, that is, belongs to $\Proj(\Rep(Q, R))$.
In fact, the set
$$\{e^{v}_{\lambda}(P) : P \in \Proj(R)\ {\rm and} \ v \in V \},$$
is a set of projective generators for the category $\Rep(Q, R)$.
On the other hand, for any injective module $E \in \Inj(R)$, the representation $ e^{v}_{\rho}(E)$ is an injective representation of $\CQ$, that is, belongs to $\Inj(\Rep(Q, R))$.
In fact, the set
$$\{e^{v}_{\rho}(E) : E \in \Inj(R) \ {\rm and}\  v \in V \},$$
is a set of injective cogenerators for the category $\Rep(Q, R)$.
This, in particular, implies that every representation $M$ of $\CQ$ can be embedded in a direct sum of elements of this set. The proof of these facts can be found in \cite{EE1} and \cite{EEG}.
\end{remark}

\section{Minimal Injective Resolutions of Path Algebras}\label{4}
In the classical representation theory of algebras over a field, to compute injective envelope and projective cover, one can  use Lemma $ 3.2.2 $  in \cite{ASS}. Enochs et al. in \cite[Proposition 3.1]{EKP} described the injective envelopes of representations of the line quivers $\overrightarrow{A_n}$ for $n \geq 1$ over an arbitrary ring. 
In this section, we plan to give an explicit formula for the terms of the minimal injective resolutions of path algebras. Unless otherwise specified, all quivers are finite, connected and acyclic.

Our first theorem, describes injective envelopes and projective precovers in the category $\Rep (\CQ, R) $ of representations of $\CQ$ by left $R$-modules. This result also can be proved by modifying the argument used in the proof of Proposition 3.1 of \cite{EKP}, where they obtained descriptions for the injective envelopes and projective precovers of representations, while $\CQ = \overrightarrow{A_n}$ is the line quiver. We here present a new proof that is in line with the proof of the main theorem of this section. Throughout we use bold capital $\BE$ to show the injective envelope.

\begin{theorem}\label{Main I}
Let $\CQ$ be an acyclic  quiver and let $M$ be a representation of $\CQ$ by left $R$-modules.
For every $v\in V$, set $K_v := \Ker (M_v \lrt \bigoplus_{s(\alpha)=v}M_{t(\alpha)})$ and $C_v := \Coker (\bigoplus_{t(\alpha)=v}M_{s(\alpha)}\lrt M_v)$. Then  $\BE(M) = \bigoplus_{v \in V} e_{\rho}^{v}(E_{v}) $, where $ E_v = \BE(K_v) $. Also $\bigoplus_{v\in V} e_{\lambda}^{v}(P_{v})$ is a projective precover of $ M $, where $ P_v $ is a projective precover of $ C_{v} $.
\end{theorem}

\begin{proof}
We only prove the first part, the proof of the second part is similar, or rather dual. To this end, we first prove that there is a monomorphism $ \psi : M \rt \bigoplus_{v \in V} e_{\rho}^{v}(E_{v}) $. Let $ w $ be a vertex and the map $ \varphi_{w}: M_w \rt E_w $ be the extension of  embedding $ K_w \hookrightarrow E_w$. If for every $ p  \in \CY_w $, we consider the composition map $ \varphi_{\Bt(p)} M_{p}:M_w\rt M_{\Bt(p)}\rt E_{\Bt(p)} $, then we get the induced map $ \psi_w=(\varphi_{\Bt(p)}M_{p})_{p \in \CY_w}:M_w \rt    ( \bigoplus_{v \in V} e_{\rho}^{v}(E_{v}))_{w} =\bigoplus_{v \in V} \bigoplus_{\CQ(w,v)} E_{v} = \bigoplus_{p\in \CY_w}E_{\Bt(p)} $.
Now, we can show that the map $ \psi =(\psi_{w})_{w \in V}: M \rt \bigoplus_{v \in V} e_{\rho}^{v}(E_{v}) $ is a morphism
Let $ \alpha:w_1 \rt  w_2$ be an arrow in $ \CQ $ and $ x \in M_{w_{1}} $. Clearly, $ \psi_{w_{2}}M_{\alpha}(x) = (\varphi_{\Bt(q)}M_{q\alpha}(x))_{q\in \CY_{w_{2}}}  $. On the other hand, $ ( \bigoplus_{v \in V} e_{\rho}^{v}(E_{v}))_{\alpha} \psi_{w_{1}} (x) = ( \bigoplus_{v \in V} e_{\rho}^{v}(E_{v}))_{\alpha} (\varphi_{\Bt(p)}M_{p}(x))_{p\in \CY_{w_{1}}} = (y_{q})_{q\in \CY_{w_{2}}}$, where

\[\begin{array}{ll} y_{q}= \left \{\begin{array}{lll} \varphi_{\Bt(p)}M_{p}(x) & {} \ \ \ \ \ p= q\alpha, \\  \\ 0  & {} \ \  \ \ {\rm otherwise.} \end{array} \right. \end{array}\]

Thus we have the following commutative diagram  for every $ \alpha:w_1\lrt w_2$   in $ \CQ $ and
therefore the map $ \psi =(\psi_{w})_{w \in V}: M \lrt \bigoplus_{v \in V} e_{\rho}^{v}(E_{v}) $ is a morphism.
\begin{center}
$\xymatrix{x \ar[rr]^{M_{\alpha}} \ar[d]^{\psi_{w_{1}}} &  &M_{\alpha}(x) \ar[d]^{\psi_{w_{2}}} \\
(\varphi_{\Bt(p)}M_{p}(x))_{p\in \CY_{w_{1}}}\ar[rr]^{ \bigoplus_{v \in V} (e_{\rho}^{v}(E_{v}))_{\alpha}} &  & \ \  (\varphi_{\Bt(q)}M_{q\alpha}(x))_{q\in \CY_{w_{2}}} }$
\end{center}

Now, we show that $\psi$ is in fact a monomorphism. Let $x \in \Ker (\psi_{w})$. Thus $\varphi_{\Bt(p)} (M_{p}(x))=0,$ for every $ p  \in \CY_w$. Assume that $\lbrace \mathcal{V}_0, \CV_1, ... ,\CV_m \rbrace$ is a partition of vertices of the subquiver $ P(\CQ)_w $, that is defined in \ref{2.3}. We show that $ M_{p}(x)=0 $ for every
$ p\in \CY_w $, where $ w \in V $. If $ p \in \CV_0 $, then the vertex $ \Bt(p) $ is a sink and so $ K_{\Bt(p)} = M_{\Bt(p)} $. Therefore $ M_{p}(x) \in K_{\Bt(p)}$ and then $ M_{p}(x)= \varphi_{\Bt(p)} (M_{p}(x))=0 $. Now assume that $ p \in \CV_1  $. Then for every arrow $ \alpha $ with initial vertex $ \Bt(p) $, $ \alpha p \in \CV_0$. Thus $ M_{\alpha}M_{p} (x)= 0$ and so $ M_{p}(x) \in K_{\Bt(p)}$. Therefore $ M_{p}(x)= \varphi_{\Bt(p)} (M_{p}(x))=0 $.
By applying this argument, after finite steps, we deduce that $ M_{\alpha} (x)=0$ for every arrow $ \alpha $ with initial vertex $ w $. Therefore $ x \in K_w $ and then $ x=\varphi_{w}(x)=0 $.

Since there is a monomorphism from $ M $ to $ \bigoplus_{v \in V} e_{\rho}^{v}(E_{v}) $, the proof of this part will be completed, if we show that $ \bigoplus_{v \in V} e_{\rho}^{v}(E_{v}) $ is an injective envelope of $ K= \bigoplus_{v\in V}s^{v}(K_v) $, where

\[\begin{array}{ll} s^v(K_v)_w= \left \{\begin{array}{lll} K_v & {\rm if} \ \ w=v , \\  \\ 0  & {\rm if}  \ \ w\neq v. \end{array} \right.\end{array}\]

To this end, it is enough to show that $ \BE(s^{v}(K_v)) =e^{v}_{\rho}(E_v)$. Note that there is a monomorphism  $ s^{v}(K_v) \lrt e^{v}_{\rho}(E_v)$.
To prove that we have an essential embedding, we need to show that $ s^{v}(K_v) \bigcap L \neq 0,$ for every nonzero subrepresentation $ L $ of $ e^{v}_{\rho}(E_v) $.
We first show that $ L_v\neq 0 $. For this, suppose to the contrary that $ L_v = 0 $. Let $ w $ be a vertex of $ \CQ $ such that $ w\neq v $ and $ L_w \neq 0 $. Also, assume that $ x = (x_{p})_{p \in \CQ(w,v)} $ is a nonzero element of $ L_{w} $. Then, there is a  path $ p \in \CQ(w,v) $ such that $ x_{p}\neq 0 $. This means that the following diagram
\[\xymatrix{L_w \ar[rr]^{L_{p}} \ar[d] & &  L_v =0 \ar[d] \\
\bigoplus_{\CQ(w,v)}E_v \ar[rr]^{(e^{v}_{\rho}(E_v))_p} & &  E_v }\]
is not commutative, which is a contradiction because $ L $ is a subrepresentation of $ e^{v}_{\rho}(E_v) $.
Therefore, $\ L_v \neq 0$.
Hence, since $ K_v $ is essential in $ E_v $, we have that $ L_v \bigcap K_v \neq 0$. So $ 0 \neq s^{v}(L_v \bigcap K_v ) \subseteq s^{v}(K_v) \bigcap L $ (Note that since $L_{t(\alpha)}=0  $, for every arrow $ \alpha $ with initial vertex $ v $, $s^{v}( L_v \bigcap K_v)$ is a subrepresentation of $ L $).
\end{proof}

As an immediate corollary, we have the following result.

\begin{corollary}
Let $ \CQ $ be a finite acyclic quiver. Every injective representation of $ \CQ $ by left $ R $-modules can be decompose into a coproduct of  representations of the form $ e^{v}_{\rho}(E) $, where $ v \in V $ and $ E \in \Inj (R).$
\end{corollary}

\begin{proof}
Any injective representation is its own injective envelope. Hence, by the above theorem every injective envelope has this form.
\end{proof}

We also need the following elementary lemma. Note that for an arrow $\alpha$ with initial vertex $w$, $\alpha \CQ(v,w) = \lbrace \alpha p \vert p \in \CQ(v,w) \rbrace$.

\begin{lemma}\label{4.1}
Let $ \CQ $  be a finite, connected and acyclic quiver. Also, assume that $  V_0 $ is the subset of $ V$ consisting of all sinks. For every $ v \in V $ and $ w \in V\setminus V_0 $  choose an arrow $ \alpha_{v,w} $ with initial vertex $ w $. Then
$ \Coker ((e^{v}_{\lambda}(R))_{w} \lrt \bigoplus_{s(\alpha)=w}(e^{v}_{\lambda}(R))_{t(\alpha)}) = \bigoplus_{\CQ''(v,w)}R $, where $ \CQ''(v,w) =  \CQ'(v,w) \setminus \alpha_{v,w} \CQ(v,w)$ and  $ \CQ' (v,w)= \bigcup_{s(\alpha)=w} \CQ(v,t(\alpha)) $.
\end{lemma}
\begin{proof}
As before, let $E_{\Bs(w)}$ denote the set of all arrows with initial vertex $w$. For every $\alpha \in E_{\Bs(w)}$,
we denote an element of $ \bigoplus_{\CQ(v,t(\alpha))}R $ by the pair
$ (z'''_\alpha , z_\alpha) $,
where
$ z_\alpha \in \bigoplus_{ \alpha \CQ(v, w)}R $, $ z'''_\alpha \in \bigoplus_{\CQ'''(v, t(\alpha))}R $
and
$ \CQ'''(v, t(\alpha)) = \CQ(v, t(\alpha)) \setminus \alpha \CQ(v, w) $.
Assume that the maps $ f: \bigoplus_{\CQ(v,w)} R \rt \bigoplus_{ \CQ'(v,w)} R $ and $ g: \bigoplus_{ \CQ'(v,w)} R \rt \bigoplus_{\CQ''(v,w)}R $
are defined by
$ f: (r_p)_{p \in \CQ(v, w)} \longmapsto (0 ,(r_p)_{\alpha p \in \alpha\CQ(v,w)} )_{s(\alpha)=w} $
and
$ g: (z'''_\alpha, z_\alpha)_{s(\alpha)=w} \longmapsto (z'''_\alpha, z_\alpha - z_{\alpha_{v,w}})_{s(\alpha)=w} $.
Now, the proof can be completed in view of the following short exact sequence
$$ 0 \lrt \bigoplus_{\CQ(v,w)} R \st {f}{\lrt} \bigoplus_{ \CQ'(v,w)} R \st{g}{\lrt} \bigoplus_{\CQ''(v,w)}R \lrt 0.$$
\end{proof}

\begin{proposition}\label{4.2}
Let $ \CQ $ be a finite, connected and acyclic quiver. If
$  0 \longrightarrow R \stackrel{}{\longrightarrow} I^{0}\stackrel{}{\longrightarrow} I^{1}\stackrel{}{\longrightarrow} I^{2} \longrightarrow \cdots $
is the minimal injective resolution of $R$, then, for every $ v \in V, $
$ 0 \lrt e^{v}_{\lambda}(R) \st{}{\lrt}  J^{v,o}\st{}{ \lrt}  J^{v,1} \st{}{\lrt}  J^{v,2} \lrt  \cdots $ is  the minimal injective resolution of $e^{v}_{\lambda}(R)$,
where for every $ i \geq 0, $  $ J^{v,i}=(\bigoplus_{x\in V_{0}}e^{x}_{\rho}(\bigoplus_{\CQ(v,x)} I^{i})) \bigoplus (\bigoplus_{y\in V\setminus V_{0}}e^{y}_{\rho}(\bigoplus_{\CQ''(v,y)}I^{i-1})) $ and $ I^{-1} = 0$.
\end{proposition}

\begin{proof}
Let $ \alpha : w \lrt w' $ be an arrow in $ \CQ $.  Set $ P^v:=e^{v}_{\lambda}(R). $ The natural morphism $ P^{v}_{\alpha}:P^{v}_{w} = \bigoplus_{\CQ(v,w)}R \lrt P^{v}_{w'}=\bigoplus_{\CQ(v, w')} R$
is a momomorphism. Hence, for every vertex $ w \in V\setminus V_0 $, the morphism $ P^{v}_{w}\lrt \bigoplus_{s(\alpha)=w} P^{v}_{t(\alpha)} $ has zero kernel.
Therefore, by Theorem \ref{Main I}, $ J^{v,0}=\BE(P^v)= \bigoplus_{x\in V}e^{x}_{\rho}(E^{v}_{x})=\bigoplus_{x\in V_0}e^{x}_{\rho}(E^{v}_{x}) = \bigoplus_{x\in V_{0}}e^{x}_{\rho}(\bigoplus_{\CQ(v,x)} I^{0})$. Note that  for every $x\in V_{0}$, $ E^{v}_{x}=\BE(P^{v}_{x})= \BE (\bigoplus_{\CQ(v,x)}R)=\bigoplus_{\CQ(v,x)}I^{0} $.

Now assume that $ K^{v,1} $ is the cokernel of the natural morphism
$ P^{v}\lrt J^{v,0} $. Let us first compute the kernel of the natural morphism
$ h^{v,1}_{w}:K^{v,1}_{w}\lrt \bigoplus_{s(\alpha)=w}K^{v,1}_{t(\alpha)} $
for every $ w\in V\setminus V_{0} $.
To this end, consider the following commutative diagram with exact rows:
\[  \xymatrix@C-0.4pc@R-0.10pc{0 \ar[r]& P^{v}_{w} \ar[r] \ar[d]^{h^{v,0}_{w}} & J^{v,0}_{w} \ar[r] \ar[d]^{g^{v,0}_{w}} & K^{v,1}_{w}  \ar[d]^{h^{v,1}_{w}} \ar[r] & 0\\
0 \ar[r] &\bigoplus_{s(\alpha)=w}P^{v}_{t(\alpha)} \ar[r] & \bigoplus_{s(\alpha)=w}J^{v,0}_{t(\alpha)} \ar[r]  & \bigoplus_{s(\alpha)=w}K^{v,1}_{t(\alpha)} \ar[r] & 0.}\]

We know that for every path $ p $ from $ w $ to a sink $ x $, there is an arrow $ \alpha_p $ as a part of $ p $ with initial vertex $ w $; i.e., $ p=p'\alpha_p $ for some path $ p' $ from $ t(\alpha_p) $ to $ x $. Thus if $z= (z_p)_{p \in (\bigcup_{x\in V_0} \CQ(w,x))} $ is an element of $\Ker(g^{v,0}_{w})$, then $ J^{v,0}_{\alpha_p}(z) = 0,$ for every $ p \in \bigcup_{x\in V_0} \CQ(w,x) $ and so $ z_p=(J^{v,0}_{\alpha_{p}}(z))_{p'} =0$.
Therefore, $ g^{v,0}_{w} $ has zero kernel.
Also, by \cite[Proposition 2.1]{EEG} this map has zero cokernel.
Hence, by the Snake lemma and Lemma \ref{4.1}, we have that
$ \Ker ( h^{v,1}_{w})\cong \Coker (h^{v,0}_{w}) \cong \bigoplus_{\CQ''(v,w)}R $.
Thus,
$ J^{v,1}=\BE(K^{v,1})=(\bigoplus_{x\in V_{0}}e^{x}_{\rho}(\bigoplus_{\CQ(v,x)} I^{1})) \bigoplus (\bigoplus_{y\in V\setminus V_{0}}e^{y}_{\rho}(\bigoplus_{\CQ''(v,y)}I^{0}))  $.

Assume that  $ J^{v,i}=(\bigoplus_{x\in V_{0}}e^{x}_{\rho}(\bigoplus_{\CQ(v,x)} I^{i})) \bigoplus (\bigoplus_{y\in V\setminus V_{0}}e^{y}_{\rho}(\bigoplus_{\CQ''(v,y)}I^{i-1})) $, $ K^{v,i+1} $ is the cokernel of the embedding $ K^{v,i} \rt J^{v,i} $ and $ \Ker (h^{v,i}: K^{v,i}_{w}\rt \bigoplus_{s(\alpha)=w}K^{v,i}_{t(\alpha)})= \bigoplus_{\CQ''(v,w)}K^{i-1} $ for every $ w \in V \setminus V_0 $, where $ i \geq 1 $ and $ K^{i-1} = \Ker (I^{i-1}\rt I^{i}) $.
Consider the following commutative diagram with the exact rows:

\[  \xymatrix@C-0.5pc@R-0.10pc{0 \ar[r]& K^{v,i}_{w} \ar[r]^{\psi^{v,i}_{w}} \ar[d]^{h^{v,i}_{w}} & J^{v,i}_{w} \ar[r] \ar[d]^{g^{v,i}_{w}} & K^{v,i+1}_{w}  \ar[d]^{h^{v,i+1}_{w}} \ar[r] & 0\\
0 \ar[r] &\bigoplus_{s(\alpha)=w}K^{v,i}_{t(\alpha)} \ar[r] & \bigoplus_{s(\alpha)=w}J^{v,i}_{t(\alpha)} \ar[r]  & \bigoplus_{s(\alpha)=w}K^{v,i+1}_{t(\alpha)} \ar[r] & 0.}\]

The map $ h^{v,i}_{w} $ is onto. Thus, by the Snake lemma, we have the following short exact sequence
$$ 0 \lrt \Ker ( h^{v,i}_{w}) \st{\overline{{\psi}^{v,i}_{w}}}{\lrt} \Ker ( g^{v,i}_{w}) \lrt \Ker ( h^{v,i+1}_{w}) \lrt 0. $$

One can easily see that $ \Ker ( g^{v,i}_{w})=\bigoplus_{\CQ''(v,w)}I^{i-1} $.
On the other hand, by the proof of the first part of Theorem \ref{Main I}, $\overline{\psi^{v,i}_{w}}(x)=x$, for every $ x \in \Ker (h^{v,i}_{w})=\bigoplus_{\CQ''(v,w)}K^{i} $.
Therefore, $ \Ker ( h^{v,i+1}_{w})\cong \Ker ( g^{v,i}_{w})/ \Ker( h^{v,i}_{w})\cong \bigoplus_{\CQ''(v,w)}K^{i} $, where $ K^i = \Coker ( d^{i-1}) $.
Hence,
$ \BE (\Ker ( h^{v,i+1}_{w}))= \bigoplus_{\CQ''(v,w)}I^{i} $ and $$ J^{v,i+1}=\BE(K^{v,i+1})=(\bigoplus_{x\in V_{0}}e^{x}_{\rho}(\bigoplus_{\CQ(v,x)} I^{i+1})) \bigoplus (\bigoplus_{y\in V\setminus V_{0}}e^{y}_{\rho}(\bigoplus_{\CQ''(v,y)}I^{i})).$$

The proof can be now completed by induction.
\end{proof}

\begin{theorem}\label{Main II}
Let $ A=R\CQ $ be the path algebra of a finite and acyclic quiver $ \CQ $. If
$ 0 \longrightarrow R \stackrel{}{\longrightarrow} I^{0}\stackrel{}{\longrightarrow} I^{1}\stackrel{}{\longrightarrow} I^{2} \longrightarrow \cdots $
is the minimal injective resolution of $R$, then the minimal injective resolution of $A$ becomes
$ 0 \lrt A \st{}{\lrt} \bigoplus_{v\in V} J^{v,o}\st{}{ \lrt} \bigoplus_{v \in V} J^{v,1} \st{}{\rt} \bigoplus_{v\in V} J^{v,2} \rt  \cdots, $
where $ J^{v,i}=(\bigoplus_{x\in V_{0}}e^{x}_{\rho}(\bigoplus_{\CQ(v,x)} I^{i})) \bigoplus (\bigoplus_{y\in V\setminus V_{0}}e^{y}_{\rho}(\bigoplus_{\CQ''(v,y)}I^{i-1})), $ for every $ i \geq 0 $ and $ I^{-1} = 0$.
\end{theorem}

\begin{proof}
We know that $ A=\bigoplus_{v\in V}e^{v}_{\lambda}(R) $. Thus the minimal injective resolution of $ A $ is the coproduct of the minimal injective resolutions of $ e^{v}_{\lambda}(R) $.
\end{proof}

\begin{example}\label{ex}
Let $0 \longrightarrow R \longrightarrow I^{0}\longrightarrow I^{1}\longrightarrow I^{2} \longrightarrow \cdots$ be a minimal injective resolution of $ R $ as a left $ R $-module and  $ \CQ $ be the line quiver $ \overrightarrow{A_n} $ with $ n \geq 1 $. In view of Proposition \ref{4.2}, we see that for every $ 1< i \leq n $,
$ 0 \lrt e^{i}_{\lambda}(R) \lrt e^{n}_{\rho}(I^0) \lrt e^{n}_{\rho}(I^1) \bigoplus e^{i-1}_{\rho}(I^0) \lrt e^{n}_{\rho}(I^2) \bigoplus e^{i-1}_{\rho}(I^1) \lrt \cdots $
is a minimal injective resolution of $ e^{i}_{\lambda}(R) $. Also $ 0 \lrt e^{1}_{\lambda}(R) \lrt e^{n}_{\rho}(I^0) \lrt e^{n}_{\rho}(I^1)  \lrt e^{n}_{\rho}(I^2)  \lrt \cdots $ is a minimal injective resolution of $ e^{1}_{\lambda}(R) $.
This, in particular, gives a minimal injective resolution for the lower triangular matrix ring of degree $n$ over $R$.
\end{example}

\begin{remark}\label{4.4}
Let $M$ be a representation of a quiver $\CQ$. We know
that there exists an exact sequence
$$ 0 \lrt \bigoplus_{\alpha\in E}e^{t(\alpha)}_{\lambda}(M_{s(\alpha)})\lrt \bigoplus_{v \in V}e^{v}_{\lambda}(M_v)\lrt M \lrt 0$$
in $\Rep (Q,R)$, see  \cite[Corollary 28.3]{Mi2} and \cite[The Standard Resolution]{C}.
On the other hand, one can easily rewrite Lemma \ref{4.1} and Proposition \ref{4.2} for every $ R $-module, instead of $R$ itself. Hence
by \cite[Corollary 1.3]{Mi4}, we can compute injective resolution for $ M $.
\end{remark}

Motivated by the theory of commutative Noetherian Gorenstein rings, Auslander introduced the notion of $ k $-Gorenstein algebras, see \cite{FGR}. Let $ k $ be a positive integer, $ R $ be a two-sided Noetherian ring and
$0 \longrightarrow R \longrightarrow I^{0}\longrightarrow I^{1}\longrightarrow I^{2} \longrightarrow \cdots$
be a minimal injective resolution of $ R $ viewed as a left $ R $-module. Then $ R $ is said to be left $ k $-Gorenstein if $ \fd (I^i)\leq i $ for every $ 0 \leq i \leq k-1 $.
Note that the notion of a $k$-Gorenstein ring is left-right symmetric \cite[Auslander's Theorem 3.7]{FGR}; i.e, $ R $ is left  $ k $-Gorenstein if and only if $ R $ is right $ k $-Gorenstein. Thus we just say that $R$ is $k$-Gorenstein if it is either left or right $ k $-Gorenstein. If $ R $ is $ k $-Gorenstein for all $ k $, then it is called an Auslander ring or sometimes an Auslander-Gorenstein ring.
Furthermore $R$ is said to be left, resp. right, quasi $k$-Gorenstein if $\fd (I^i)\leq i +1$ for every $0 \leq i \leq k-1$, see \cite{Hu}. Note that the notion of quasi $k$-Gorenstein ring is not left-right symmetric, see \cite{AR2}. If $R$ is left and right quasi $k$-Gorenstein, we say that $R$ is quasi $k$-Gorenstein.

\begin{corollary}
Let $ \CQ $ be a finite and acyclic quiver. If $ R $ is $ k$-Gorenstein,
then $ R\CQ $ is quasi $ k $-Gorenstein.
\end{corollary}
\begin{proof}
Without loss of generality we can assume that $ \CQ $ is a connected quiver.
Let
$0 \longrightarrow R \longrightarrow I^{0}\longrightarrow I^{1}\longrightarrow I^{2} \longrightarrow \cdots$
be a minimal injective resolution of $ R $ such that $ \fd (I^i) \leq i$ for every $ 0 \leq i \leq k-1 $.
The Standard Resolution in \ref{4.4}
easily implies that if $ I \in \RMod $ is of flat dimension at most $n$, then
$ e^{v}_{\rho}(I) $, as an object in $ \Rep (\CQ, R) $, has flat dimension at most $n+1$.
This result follows from the fact that the first two left terms of this
sequence has flat dimension at most $n$.
Now it follows easily from Theorem \ref{Main II} that  $ \fd (\bigoplus_{v \in V}J^{v,i})\leq i+1 $ for every $ v \in V $ and for all $ 0\leq i \leq k-1 $.
\end{proof}

In \cite{N} Nakayama proposed a conjecture, which by results of Muller \cite{Mu} is equivalent to the following:
Let $ R $ be a finite dimensional algebra over a field $ K $ and let $0 \longrightarrow R \longrightarrow I^{0}\longrightarrow I^{1}\longrightarrow I^{2} \longrightarrow \cdots$ be a minimal injective resolution of $ R $ as a left $ R $-module. If all the $ I^i $  are projective, then $ R $ is self-injective. Later Auslander and Reiten in \cite{AR} proposed a generalized version of this conjecture which is called Generalized Nakayama Conjecture or for simplicity $\mathbf{GNC}$.
This conjecture says that  each indecomposable injective module is a summand of some terms in a minimal injective resolution of an artin algebra.
In the following corollary we show that If $ \mathbf{GNC} $ is true for $ R $, then $\mathbf{GNC}$ is true for $ R\CQ $, where $\CQ$ is a finite and acyclic quiver.

\begin{corollary}\label{GNC}
Let $ \CQ $ be a finite and acyclic quiver. If $ \mathbf{GNC} $ is true for $ R $, then $\mathbf{GNC}$ is true for $ R\CQ $.
\end{corollary}

\begin{proof}
Without loss of generality we can assume that $ \CQ $ is a connected quiver.
Let
$0 \longrightarrow R \longrightarrow I^{0}\longrightarrow I^{1}\longrightarrow I^{2} \longrightarrow \cdots$
be a minimal injective resolution of $ R $. By assumption, every indecomposable injective module appears as a summand of some $ I^{n}.$
Note that every indecomposable injective $ R\CQ $-module is of the form $ e^{v}_{\rho}(I) $, where $ v $ is a vertex of $ \CQ $ and $ I $ is an indecomposable injective $ R $-module.
By Theorem \ref{Main II}, we know that  for every $ v \in V $ and $ i \geq 0 $, $ e^{v}_{\rho}(I^i) $ appears in the minimal injective resolution of $ R\CQ $ as a summand of some terms.
If $ v \in V_0 $, then $ e^{v}_{\rho}(I^i) $ is a direct summand of $ J^{v,i} $ for every $ i\geq 0 $.
If $ v \in V \setminus V_0 $, then there is a vertex $ w \in V_{\Bs(v)} $. A trivial verification shows that $ e^{v}_{\rho}(I^i) $ is a direct summand of $ J^{w,i+1} $ for every $ i \geq 0 $. \end{proof}

\section{Auslander-Gorenstein Property of Path Algebras}\label{5}
It was proved by Iwanaga and Wakamatsu in \cite[Theorem 8]{IW} that a left and right Artinian ring $ R $ is a $ k $-Gorenstein ring if and only if so is the lower triangular matrix ring $ T_{n}(R) $ of degree $ n $ over $ R $. Observe that this is a generalization of \cite[Theorem 3.10]{FGR}, where the case $ n=2 $ was established.
In this section, we study the  $ k $-Gorensteiness of $ R\CQ $, where $ \CQ $ is a finite, connected and acyclic quiver. 
We know that when $ \CQ $ is a linear quiver $ \overrightarrow{A_n} $, the algebra $ R\CQ $ is isomorphic to the lower triangular matrix ring of degree $ n $ over $ R $. 

\begin{lemma}
Let $ \CQ=\overrightarrow{A_{n}} $ and $ R $ be a $ k $-Gorenstein ring for some positive integer $ k $. Then $ R\CQ $ is a $ k$-Gorenstein ring.
\end{lemma}

\begin{proof}
Let
$0 \longrightarrow R \longrightarrow I^{0}\longrightarrow I^{1}\longrightarrow I^{2} \longrightarrow \cdots$
be a minimal injective resolution of $ R $ such that $ \fd (I^i) \leq i$ for every $ 0 \leq i \leq k-1 $.
Consider $0 \longrightarrow R\CQ \longrightarrow J^{0}\longrightarrow J^{1}\longrightarrow J^{2} \longrightarrow \cdots$
as a minimal injective resolution of $ R\CQ $.
On the other hand, by Example \ref{ex}, we know that
$ J^i=(\bigoplus_{n}e^{n}_{\rho}(I^i)) \bigoplus (e^{1}_{\rho}(I^{i-1})\bigoplus \cdots \bigoplus e^{n-1}_{\rho}(I^{i-1}) )$ for every $ i\geq 0 $, where $ I^{-1}=0 $. Note that  $ \fd (e^{n}_{\rho}( I^{i}))=\fd (I^i) \leq i $.
Also for every $1\leq j \leq n-1$, $\fd (e^{j}_{\rho}(I^{i-1}))\leq \fd (I^{i-1})+1\leq (i-1)+1=i $. Therefore, the path algebra $ R\CQ $ is a $ k$-Gorenstein ring.
\end{proof}

\begin{lemma}\label{5.2}
Let $\CQ $ be a finite, connected and acyclic quiver with $ n$ vertices. If the path algebra $ R\CQ $ is a  $ k $-Gorenstein ring, then $ \CQ$ is the linear quiver $\overrightarrow{A_{n}}$ and $ R $ is $ k $-Gorenstein.
\end{lemma}

\begin{proof}
We first show that $ \CQ $ is a linear quiver $ \overrightarrow{A_{n}} $. To this end, suppose to the contrary that $ \CQ\neq \overrightarrow{A_n} $. Thus, there is a vertex $v$ in $ \CQ $ such that
$ \vert E_{\Bt(v)}\vert >1 $ or $ \vert E_{\Bs(v)}\vert >1 $; i.e., there exist more than one arrows ending at $ v $ or begining at $ v $.
In the case $ \vert E_{\Bt(v)}\vert >1 $, assume that $ \alpha_{1}, \cdots, \alpha_{m} $ are all the arrows with terminal vertex $ v $, where $ m $ is an integer more than $ 1 $. Also suppose that $ s(\alpha_{i})=v_{i} $ for every $ i=1,... ,m $. Since $ \CQ $ is a connected and acyclic quiver, there is a sink vertex $ w\in V_{0} $ such that $ \CQ(v,w)\neq \emptyset $. Thus we can consider the following diagram as a part of quiver $ \CQ $.
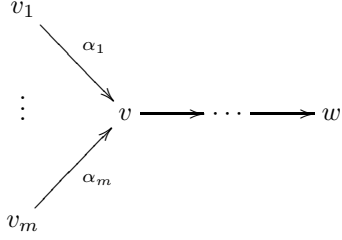
\[\xymatrix{v_{1}\ar[rd]^{\alpha_{1}}   &  & & \\  \vdots & v \ar[r]  &\ar[r] \cdots & w  \\ v_{m}\ar[ru]_{\alpha_{m}}  &  & & }\]

We know that for a vertex $ x $, $ e^{w}_{\lambda}(R)_{x} = \bigoplus_{\CQ(w,x)}R $ is $ R $ if $ x=w $ and is zero otherwise. Also, for every module $I$,
$ e^{x}_{\rho}(I)_{w}= \prod_{\CQ(w,x)}I $ is $ I $ if $ x=w $ and is zero otherwise. Thus, $ e^{w}_{\rho}(I^{0}) $ is a direct summand of $ \BE(e^{w}_{\lambda}(R)) $, where $ I^0 = \BE (R) $. See also Theorem \ref{Main II}.
Clearly, the natural morphism
$ \bigoplus_{1\leq i \leq m}( \prod_{\CQ(v_{i},w)}I^0) \lrt \prod_{\CQ(v,w)}I^0 $
is not injective. Therefore, by \cite[Proposition 3.4]{EOT},  $ e^{w}_{\rho}(I^0) $ is not flat.
Hence, $ \BE(e^{w}_{\lambda}(R))$ is not flat.
From this we can conclude that $R\CQ$ is not
$ k $-Gorenstein, a contradiction.

Now consider the case where $ \vert E_{\Bs(v)}\vert >1 $. In view of the above argument, we may assume that $ \vert E_{\Bt(w)}\vert \leq1 $ for every vertex $w$; i.e., for every vertex $ w $ in $ \CQ $ there is at most one arrow with terminal vertex $ w $. We show that in this case, $ e^{x}_{\rho}(I) $ is not flat, for every non-zero $ R $-module $ I $ and every $ x \in V $. This contradicts the fact that $ R\CQ $ is $ k $-Gorenstein.
In this case, assume that $ \alpha_{1}, \cdots, \alpha_{m} $ are all the arrows with initial vertex $ v $, where $ m $ is an integer more than $ 1 $. Also, suppose that $ s(\alpha_{i})=v_{i} $ for every $ i=1,... ,m $. Thus, we can consider the following diagram as a part of quiver $ \CQ $
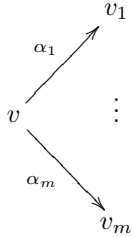
\[\xymatrix{& v_{1}\\ v \ar[ur]^{\alpha_{1}} \ar[dr]_{\alpha_{m}}   & \vdots \\  & v_{m} }\]

We need to consider the following two cases:

$(i) $ Let $ \CQ(v,x) \neq \emptyset$. Since for every vertex $ w $ in $ \CQ $ there is at most one arrow with terminal vertex $ w $, we can conclude that $\vert \CQ(v,x ) \vert=1 $. Hence, we can assume that $ \CQ(v,x )=\lbrace p \rbrace$ for some path $p$ with initial vertex $v$ and terminal vertex $x$. Therefore, there is an integer $ i $ such that $ 1\leq i \leq m $ and $ \alpha_{i} $ is a part of the path $ p $. Thus, for every $ j\neq i $ there is no path from $ v_{j} $ to $ x $.
Clearly the natural morphism
$ \bigoplus_{t(\alpha)=v_j}e^{x}_{\rho}(I)_{s(\alpha)}\lrt e^{x}_{\rho}(I)_{v_j} $ is not injective, because $e^{x}_{\rho}(I)_{v_{j}}=0$  and $e^{x}_{\rho}(I)_{v}=I\neq 0$. Therefore $e^{x}_{\rho}(I)$ is not flat.

$(ii)$ Let $\CQ(v,x) = \emptyset$. In this case the proof falls naturally into two parts:

$(1)$ Let $\CQ(x,v)\neq\emptyset$. Hence, we can assume that $ \CQ(x,v )=\lbrace p \rbrace$ for some path $ p $
beginning at vertex $ x $ and ending at vertex $ v $. If the arrow $ \beta $ is a part of the path $p$ with initial vertex $ x $, then we see that the natural morphism  $ \bigoplus_{t(\alpha)=t(\beta)}e^{x}_{\rho}(I)_{s(\alpha)}\lrt e^{x}_{\rho}(I)_{t(\beta)} $ is not injective, because $ e^{x}_{\rho}(I)_{t(\beta)}=0 $  and $ e^{x}_{\rho}(I)_{x}=I\neq 0 $.

$(2) $ Let $ \CQ(x,v) = \emptyset$. Since $ \CQ $ is connected, $ \CQ(v,x)=\emptyset $ and $ \CQ(x,v)=\emptyset $, there is a vertex $ w $ such that $ \CQ(w,v) \neq \emptyset$ and $ \CQ(w,x)\neq\emptyset $.
We can assume that $ \CQ(w,v )=\lbrace p \rbrace$  for some path $ p $ with initial vertex $ w $ and terminal vertex $ v $. If the arrow $ \beta $ is a part of the path $p$ such that $ \CQ(s(\beta),x)\neq \emptyset $ and $ \CQ(t(\beta),x)=\emptyset $, the natural morphism  $ \bigoplus_{t(\alpha)=t(\beta)}e^{x}_{\rho}(I)_{s(\alpha)}\lrt e^{x}_{\rho}(I)_{t(\beta)} $ is not injective, because $ e^{x}_{\rho}(I)_{t(\beta)}=0 $  and $ e^{x}_{\rho}(I)_{s(\beta)}=I\neq 0 $.

Now the proof is completed by showing that $ R $ is a $ k $-Gorenstein ring. By the above, we know that $ \CQ=\overrightarrow{A_n} $. Let
$  0 \longrightarrow R \longrightarrow I^{0}\longrightarrow I^{1}\longrightarrow I^{2} \longrightarrow \cdots $
be the minimal injective resolution of $R$. Then
$0 \lrt e^{1}_{\lambda}(R)\lrt e^{n}_{\rho}(I^0)\lrt e^{n}_{\rho}(I^1)\lrt e^{n}_{\rho}(I^2)\lrt \cdots $
is the minimal injective resolution of $e^{1}_{\lambda}(R)$. Since $ R\CQ $ is a $ k $-Gorenstein ring, we conclude that $\fd (I^i) \leq \fd (e^{n}_{\rho}(I^i))\leq i$ for every $ 0\leq i\leq k-1 $, and the proof is complete.
\end{proof}

We summarize the above results in the following theorem.

\begin{theorem}\label{Main III}
Let $\CQ$ be a finite, connected and acyclic quiver with $n$ vertices. Then the path algebra $R\CQ$ is a $k$-Gorenstein ring if and only if $\CQ=\overrightarrow{A_{n}}$ and $ R $ is a $ k $-Gorenstein ring.
\end{theorem}

Let $M$ be an $R$-module. $M$ is said to have the dominant dimension at least $ n \in \N $, written $\rm{dom.dim}\ M\geq n$, if there exists a minimal injective resolution
$0 \longrightarrow M \longrightarrow I^{0}\longrightarrow I^{1}\longrightarrow I^{2} \longrightarrow \cdots$
of $M$ such that all the modules $I^j$ with $0 \leq j \leq n-1$ are projective-injective.
If the injective envelope $ I^{0} $ of M is not projective, we set $\rm{dom.dim}\ M = 0$. In the
case $\rm{dom.dim}\ M\geq n$ and $\rm{dom.dim}\ M\ngeq n+1$, we say $\rm{dom.dim}\ M= n$. If no
such $n$ exists, we write $\rm {dom.dim}\ M= \infty$ .
Using our techniques, we can prove the following theorem as a generalization of Theorem 3.6 of \cite{A1}.

\begin{theorem}
Let $ \CQ $ be a finite, connected and acyclic quiver  with $ n > 1 $ vertices. Then
\[\begin{array}{ll} \rm{dom.dim}\ R\CQ= \left \{\begin{array}{lll} 1 & {\rm if} \ \ \CQ=\overrightarrow{A_n}\ {\rm and} \ \rm{dom.dim}\ R\neq 0 , \\  \\ 0  &  \rm{otherwise.} \end{array} \right.\end{array}\]
\end{theorem}

\begin{proof}
Let
$0 \longrightarrow R \longrightarrow I^{0}\longrightarrow I^{1}\longrightarrow I^{2} \longrightarrow \cdots$
be a minimal injective resolution of $ R $ such that $ \fd (I^i) \leq i$ for every $ 0 \leq i \leq k-1 $.
Consider $0 \longrightarrow R\CQ \longrightarrow J^{0}\longrightarrow J^{1}\longrightarrow J^{2} \longrightarrow \cdots$
as a minimal injective resolution of $ R\CQ $. We consider the following two cases: $ (i) $ $ \CQ\neq\overrightarrow{A_n}  $. By Lemma \ref{5.2}, $ J^0 $ is not projective and  hence $ \rm{dom.dim}\ R\CQ = 0 $. $ (ii) $ $ \CQ = \overrightarrow{A_n}  $. By Example \ref{ex}, we know that
$ J^i=(\bigoplus_{n}e^{n}_{\rho}(I^i)) \bigoplus (e^{1}_{\rho}(I^{i-1})\bigoplus \cdots \bigoplus e^{n-1}_{\rho}(I^{i-1}) )$ for every $ i\geq 0 $, where $ I^{-1}=0 $. If $ \rm{dom.dim}\ R = 0 $, then  $ J^0 $ is not projective and hence $ \rm{dom.dim}\ R\CQ = 0 $. But if $ \rm{dom.dim}\ R\CQ \neq 0 $, then $ J^0 $ is projective. Hence,  $ \rm{dom.dim}\ R\CQ \geq 1 $. On the other hand, $ J^1 $ is not projective. Thus, $ \rm{dom.dim}\ R\CQ \ngeq 1 $. Therefore $ \rm{dom.dim}\ R\CQ = 1 $.
\end{proof}

Observe that one of the advantage of working in the the category of representations, with value in the category of $R$-modules when $R$ is an arbitrary ring (not only field) is to study the category of representations of a quiver $\CQ$ with relations over a field. The following example is devoted to this and shows the power of Theorem \ref{Main III}.

Let us first recall briefly the notion of quivers with relations. Let $\CQ$ be an arbitrary quiver. A relation in $\CQ$ with coefficients in $R$ is an $R$-linear combination of paths of length at least two having the same source and target.
A relation usually is denoted by $ \rho = \sum_{i=1}^{m} r_i\gamma_i $, where $ r_i \in R $ and $\gamma_i$ are paths of $ \CQ $ of length at least $ 2 $ having a common source and a common target. 

If $ I $ is a set of relations for a quiver $ \CQ $ such that the ideal it generates $ \langle I \rangle$ contains all paths
of length at least $ m $ for some positive integer $ m \geq 2 $, then $ \langle I \rangle $ is called admissible. 

Let $ I $ be a set of relations such that the ideal $ \langle I \rangle$ is an admissible ideal. Then $ \Rep (\CQ_I, R) $ denotes the full subcategory of $ \Rep (\CQ, R) $ consisting of
the representations of $ \CQ $ bound by $ \langle I \rangle $; i.e.,
all representations $ M $ such that $ M_\rho= \sum_{i=1}^{m} r_i M_{\gamma_i} =0 $, for every relation $ \rho = \sum_{i=1}^{m} r_i\gamma_i $ in $ I $. Also, $ \Rep (\CQ_I, R) $ is equivalent to $ \RQIMod $. 

\begin{example}\label{5.6}
$ (i) $ Let $ K $ be a field and $\CQ$ be the quiver
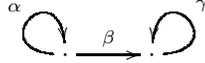
\[\xymatrix{ & \cdot  \ar@(l,u)[]^{\alpha} \ar[r]^{\beta} & \cdot \ar@(r,u)[]_{\gamma} }\]
Set $I=\{ {\alpha}^2,{\gamma}^2, \beta\alpha-\gamma\beta\}$. Then the category $\Rep(\CQ_I, K)$ is equivalent to the category $\Rep(\CQ',R)$, where
$\CQ'$ is the quiver
\[\xymatrix{ \cdot \ar[r]& \cdot}\]
and $R=K[x]/(x^2)$.
Now, by our results, it is easily seen  that $R\CQ'$ is an Auslander ring and $ \rm{dom.dim}\ R\CQ'=1 $.

$ (ii) $ As another example, assume that $ R $ is the $ K $-algebra given by the quiver
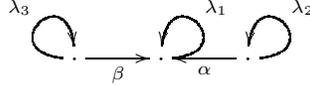
\[\xymatrix{ & \cdot \ar@(l,u)[]^{\lambda_{3}} \ar[r]_{\beta} & \cdot \ar@(r,u)[]_{\lambda_{1}}  & \cdot \ar[l]^{\alpha}\ar@(r,u)[]_{\lambda_2} }\]
with relations ${ \lambda_{1}}^{2}, { \lambda_{2}}^{2}, { \lambda_{3}}^{2}, \alpha\lambda_2-\lambda_1\alpha, \beta\lambda_3-\lambda_1\beta $. Then the category $ \Rep (\CQ_I, K)$ is equivalent to the category $ \Rep (\CQ',R') $, where $ \CQ' $ is a quiver $\cdot \lrt \cdot \longleftarrow \cdot$ and $R'=K[x]/(x^2)$. Now, by our results, it is easily seen  that $R'\CQ'$ is not  $ 1 $-Gorenstein  and $ \rm{dom.dim}\ R'\CQ'=0 $. Also one can easily get a minimal injective resolution for $R'\CQ'$.
\end{example}

\subsection{Auslander-Gorenstein Property of Tensor Product of Path Algebras}\label{6}
This subsection is devoted to tensor product of path algebras. For quivers $ \CQ $ and $ \CQ' $ we define the tensor product quiver $\CQ\otimes \CQ'$ by $(\CQ \otimes \CQ')_0= \CQ_0 \times \CQ'_0$ and $(\CQ \otimes \CQ')_1= (\CQ_0 \times \CQ'_1) \cup (\CQ_1 \times \CQ'_0)$, where the maps $ t,\ s: (\CQ \otimes \CQ')_1 \lrt (\CQ \otimes \CQ')_0 $ are defined by $ t(\alpha, \ b)= (t(\alpha), \ b), \ t( a, \ \beta)= (a , \  t(\beta)),\ s(\alpha, \ b)= (s(\alpha),\ b ),$ and $ s(a, \ \beta)=(a, \ s(\beta)) $ for all $ a \in \CQ_0, \ b \in \CQ'_0, \ \alpha \in \CQ_1,$ and $ \beta \in \CQ'_1$.

Now let $ K $ be a field and $ I $ (resp. $ I' $) be a set of relations in $ \CQ $ (resp. $ \CQ' $). It is proved in \cite[Lemma 1.3]{Les} that there is a $K$-algebra isomorphism
$ K\CQ/ \CI  \otimes_K K\CQ'/ \CI'  \cong K(\CQ \otimes \CQ') / \CI \Box \CI',$
where $ \CI = \langle I \rangle$,  $\CI' = \langle I' \rangle  $, and
$\CI \Box \CI'$ is an ideal of $K(\CQ \otimes \CQ')$ generated by the set $ I \Box I' $ consisting of $(\CQ_0 \times I') \cup (I \times \CQ'_0)$ and all differences
$(\alpha,w') \circ (v, \beta) - (w, \beta) \circ (\alpha, v')$
in which $\alpha:v\rt w$ and $\beta:v'\rt w'$ are arrows in $\CQ_1$ and $\CQ'_1$, respectively.
Hence, there is an  $K$-algebra isomorphism $K(\CQ \otimes \CQ')/\CI \Box \CI' \cong K(\CQ' \otimes \CQ)/\CI'\Box \CI.$

By the above discussion we are able to prove a useful technical result.

\begin{stheorem}\label{tensor}
There exist the following equivalences of categories
\[\begin{array}{ll}
\Rep(\CQ \otimes \CQ'_{I\Box I'}, K) & \simeq \Rep (\CQ_I, K\CQ'/\CI')\\
& \simeq \Rep(\CQ'_{I'}, K\CQ/\CI).
\end{array}\]
\end{stheorem}

\begin{proof}
We prove the first equivalence, the second one follows in view of the remark above.
To this end, we define two functors $F: \Rep(\CQ \otimes \CQ'_{I\Box I'}, K) \lrt \Rep (\CQ_I, K\CQ'/\CI')$ and $G: \Rep (\CQ_I, K\CQ'/\CI') \lrt \Rep(\CQ \otimes \CQ'_{I\Box I'}, K)$ that are quasi-inverse.
First let us interpret the tensor product quiver $\CQ \otimes \CQ'$ with a set of relations $I\Box I'$. Replace each vertex of $\CQ$ by a copy of quiver $\CQ'$. Then, for the set of arrows, add arrows corresponding to the set $(\CQ_1 \times \CQ'_0)$. In addition, the set of relations $I \Box I'$ consists of the set of relations $I$, $I'$ and the commutativity relations that induced from addition arrows corresponding to $(\CQ_1 \times \CQ'_0)$.

Now, the functor $F:  \Rep(\CQ \otimes \CQ'_{I\Box I'}, K) \lrt \Rep (\CQ_I, K\CQ'/\CI')$ is defined as follows. Given a representation $\CM \in  \Rep(\CQ \otimes \CQ'_{I\Box I'}, K)$, $F(\CM)_v$ is a representation corresponding to a copy of quiver $\CQ'$ in a vertex $v$. Moreover, for every arrow $\alpha: v\lrt w$ of $\CQ$, $F(\alpha): F(\CM)_v \lrt F(\CM)_w$ is a morphism corresponding to the set $(\CQ_1 \times \CQ'_0)$. Note that the commutativity relations that is introduced above, implies that $F(\alpha)$ is a morphism of $K\CQ'/\CI'$-modules. Indeed, for every arrow we have a natural transformation between bound quiver representations of $ \CQ' $ by vector spaces.
The same method as above can be applied to define a functor
$ G: \Rep (\CQ_I, K \CQ'/\CI') \lrt \Rep(\CQ \otimes \CQ'_{I\Box I'}, K) $ that will be a  quasi-inverse of $F$.
\end{proof}

Let $ A=K\CQ $ and $ B= K\CQ' $ be path algebras with respect to finite, connected and acyclic quivers $\CQ$ and $ \CQ'$, where $ K $ is a field.  By Theorem \ref{tensor} and \cite[Lemma 1.3]{Les}, we can easily see that $ A\otimes_{K} B \cong (K\CQ)[\CQ']$. Thus one can compute easily
minimal injective resolution of the tensor product of path algebras $ A\otimes_{K} B$. Moreover, by Theorem \ref{Main III},
we have the following corollary.

\begin{scorollary}\label{tensorprod}
Let $ k $ be a positive integer and $ A=K\CQ $ and $ B= K\CQ' $ be path algebras of finite, connected and acyclic quivers $\CQ$ and $ \CQ'$, where $ K $ is a field.  Then
$ A\otimes_{K} B$ is $ k $-Gorenstein if and only if $ \CQ $ and $ \CQ' $ are linear quivers.
\end{scorollary}

\section*{Acknowledgments}
The authors would like to thank the referee for useful comments and hints that improved our exposition. The authors also thank the Center of Excellence for Mathematics (University of Isfahan).


\begin{thebibliography}{9999}
\bibitem{A1}
M. Abrar, Dominant dimensions of two classes of finite dimensional algebras, preprint (2012), available at http://arxiv.org/pdf/1209.0562v1.pdf.

\bibitem{AEHS}
J. Asadollahi, H. Eshraghi, R. Hafezi and Sh. Salarian, On the homotopy categories of projective
and injective representations of quivers, J. Algebra {\bf 346} (2011), 101-115.


\bibitem{ASS}
I. Assem, D. Simson and A. Skowronski, Elements of the representation theory of associative algebras. Vol. 1. Techniques of representation theory. London Mathematical Society Student Texts, {\bf 65}, Cambridge University Press, Cambridge, 2006.

\bibitem{AR}
M. Auslander and I. Reiten, On a generalized version of the Nakayama conjecture, Proc. Amer. Math. Soc. {\bf 52} (1975), 69-74.

\bibitem{AR1}
M. Auslander and I. Reiten, k-Gorenstein algebras and syzygy modules, J. Pure Appl. Algebra. {\bf 92(1)} (1994), 1-27.

\bibitem{AR2}
M. Auslander and I. Reiten, Syzygy modules for Noetherian rings, J. Algebra. 183 (1) (1996), 167-185.

\bibitem{C}
W. Crawley-Boevey, Lectures on representations of quivers, A graduate course given in 1992 at Oxford University, available at http://www1.maths.leeds.ac.uk/~pmtwc/quivlecs.pdf.

\bibitem{EE1}
E. Enochs and S. Estrada, Projective representations of quivers, Comm. Algebra. {\bf 33} (2005), 3467-3478.

\bibitem{EEG}
E. Enochs, S. Estrada and J. R. Garcia Rozas, Injective representations of infinite quivers. Applications, Canad. J. Math. {\bf 61} (2009), 315-335.



\bibitem{EGOP}
E. Enochs, J. R. Garcia Rozas, L. Oyonarte and S. Park, Noetherian Quivers, Quaest. Math. {\bf 25(4)} (2002), 531-538.

\bibitem{EH}
E. Enochs and I. Herzog, A homotopy of quiver morphisms with applications to representations, Canad. J. Math. {\bf 51(2)} (1999), 294-308.

\bibitem{EKP}
E. Enochs, H. Kim and S. Park, Injective Covers and Envelopes of Representations of Linear Quivers, Comm. Algebra. {\bf 37(2)} (2009), 515-524.

\bibitem{EOT}
E. Enochs, L. Oyonarte and B. Torrecillas, Flat covers and flat representations of quivers, Comm. Algebra. {\bf 32(4)} (2004), 1319-1338.

\bibitem{E}
S. Estrada, Monomial algebras over infinite quivers, Applications to $N$-complexes of modules, Comm. Algebra. {\bf 35(10)} (2007), 3214-3225.

\bibitem{EHIS}
K. Erdmann, T. Holm, O. Iyama and J. Schröer, Radical embedding and representation dimension, Adv. Math. {\bf 185} (2004), 159-177.

\bibitem{FGR}
R. M. Fossum, P. A. Griffith and I. Reiten, Trivial extensions of abelian categories, Lecture Notes in Math., Vol. {\bf 456}, Springer-Verlag, Berlin, 1975.

\bibitem{G1}
P. Gabriel, Unzerlegbare Darstellungen I, Manuscripta Math. {\bf 6} (1972), 71-103.

\bibitem{G2}
P. Gabriel, Indecomposable representations II, Symposia Math. Inst. Naz. Alta Mat. {\bf 11} (1973), 81-104.


\bibitem{Hu}
Z. Y. Huang, Syzygy modules for quasi k-Gorenstein rings, J. Algebra. {\bf 299} (2006), 21-32.



\bibitem{IW}
Y. Iwanaga and T. Wakamatsu, Auslander-Gorenstein property of triangular matrix rings, Comm. Algebra. {\bf 23} (1995), 3601-3614.


\bibitem{Le}
L. Le Bruyn and C. Procesi, Semisimple representations of quivers, Trans. Amer. Math. Soc. {\bf 317(2)} (1990), 585-598.

\bibitem{Les}
Z. Leszczy\'{n}ski, On the representation type of tensor product algebras, Fund. Math. {\bf 144} (1994), 143-161.

\bibitem{LM}
S. Liu and J. P. Morin, The strong no loop conjecture for special biserial algebras, Proc. Amer. Math. Soc. {\bf 132} (2004), 3513-3523.

\bibitem{Mi1}
B. Mitchell, On the dimension of objects and categories II, J. Algebra. {\bf 9} (1968), 341-368.

\bibitem{Mi2}
B. Mitchell, Rings with several objects, Adv. Math. {\bf 8} (1972), 1-161.

\bibitem{Mi4}
J. Miyachi, Injective resolutions of Noetherian rings and cogenerators, Proc. Amer. Math. Soc. {\bf 128} (2000), 2233-2242.

\bibitem{Mu}
B. J. Muller, The classification of algebras by dominant dimension, Canad. J. Math. {\bf 20} (1968), 398-409.

\bibitem{N}
T. Nakayama, On algebras with complete homology, Abh. Math. Sem. Univ. Hamburg {\bf 22} (1958), 300-307.

\end{thebibliography}
\end{document}